\documentclass[11pt,twoside,reqno]{amsart}
\usepackage{amsmath, amsfonts, amsthm, amssymb,amscd, graphicx, amscd}

\usepackage{movie15}
\usepackage{cancel}
\usepackage{float,epsf}
\usepackage[english]{babel}
\usepackage{enumerate}
\usepackage{tikz}
\usepackage{hyperref}

\usepackage{srcltx}

\usepackage{geometry}
\usepackage{verbatim}
\usepackage{mathrsfs}

\graphicspath{ {11./images/} }

\newtheorem{theorem}{Theorem}[section]

\newtheorem{question}[theorem]{Question}

\newtheorem*{remark*}{Remark}

\numberwithin{equation}{section}
\numberwithin{figure}{section}

\DeclareMathOperator{\sech}{sech}

\def\intave#1{\int_{#1}\hbox{\llap{$\raise2.3pt\hbox{\vrule
height.9pt width7pt}\phantom{\scriptstyle{#1}}\mkern-2mu$}}}

\begin{document}
\title{A stronger constant rank theorem}
\author{Qinfeng Li, \quad Lu Xu}
\address{School of Mathematics, Hunan University, Changsha, 410082,
Hunan Province, P.R. China}
\email{liqinfeng@hnu.edu.cn}
\address{School of Mathematics, Hunan University, Changsha, 410082,
Hunan Province, P.R. China}
\email{xulu@hnu.edu.cn}

\thanks{Research of the first author is supported by National Key R$\&$D Program of China (2022YFA1006900) and the
National Science Fund for Youth Scholars (No. 1210010723). Research of the second author is supported by NSFC NO. 12171143.}

\begin{abstract}
Motivated from one-dimensional rigidity results of entire solutions to Liouville equation, we consider the semilinear equation
\begin{align}
    \label{liouvilleequationab}
\Delta u=G(u) \quad \mbox{in $\mathbb{R}^n$},
\end{align}where $G>0, G'<0$ and $GG^{''}\le A(G')^2$, with $A>0$. Let $u$ be a smooth convex solution and $\sigma_k(D^2 u)$ be the $k$-th elementary symmetric polynomial with respect to $D^2u$. We prove stronger constant rank theorems in the following sense:
\begin{enumerate}
    \item When $A\le  2$, if $\sigma_2(D^2u)$ takes a local minimum, then $D^2 u$ has constant rank $1$.
    \item  When $A\le \frac{n}{n-1}$, if $\sigma_n(D^2 u)$ takes a local minimum, then $\sigma_n(D^2 u)$ is always zero in the domain.
\end{enumerate}

\end{abstract}

\maketitle
\section{Introduction}
\subsection{Motivation}
The paper is motivated in the study of classification of entire solutions to the Liouville equation
\begin{align}
\label{liouvilleequation}
    -\Delta u=e^{2u}, \quad \mbox{in $\mathbb{R}^n$}.
\end{align}
When $n=2$, \eqref{liouvilleequation} describes planar conformal metrics of Gaussian curvature equal to $1$, and when $n\ge 3$, the equation also has its own interest since it arises in the theory of gravitational equilibrium of polytropic stars, and classification of solutions to \eqref{liouvilleequation} is crucial
in the study of the corresponding problems in bounded domains, as pointed in Farina \cite{Farina}.

Due to the seminal work of Chen-Li\cite{CL}, it is known that when $n=2$ and if the solution $u$ satisfies $\int_{\mathbb{R}^n}e^{2u}\, dx <\infty$, then $u$ is radial about a point. A different classification result is due to Dancer-Farina \cite{DF}, where it is shown that there are no finite Morse index solutions to \eqref{liouvilleequation} when the dimension is between $3$ and $9$, and when dimension is equal to $2$, finite Morse index solutions must be radial about a point. Recently in \cite{EGLX}, the authors prove several rigidity results based on other various geometric and asymptotic assumptions, and the following new question was proposed:
\begin{question}
\label{concavity}
Are all concave solutions to \eqref{liouvilleequation} necessarily one-dimensional?
\end{question}

The motivation of Question \ref{concavity} comes from the following connection between the Liouville equation \eqref{liouvilleequation} and the De Giorgi conjecture on phase transitions. A one dimensional solution to \eqref{liouvilleequation} looks like $u(x',x_n)=\ln(\sech x_n)$ up to translation, scaling and rotation, where $x'=(x_1,\cdots, x_{n-1})$. Interestingly, $$\frac{\partial u}{\partial x_n}=-\tanh x_n,$$ which up to scaling is the one-dimensional solution to the classical Allen-Cahn equation
       \begin{align}
       \label{acequation}
           \Delta u=u^3-u,\quad \mbox{in $\mathbb{R}^n$}.
       \end{align}
The De Giorgi conjecture says that if along a direction, the directional derivative of the solution to Allen-Cahn equation does not change sign, then the solution must be one-dimensional. Considering the negative trace of the hessian of solution to \eqref{liouvilleequation}, a natural assumption in the case of Liouville equation is thus the negativity of hessian of the solution, from which Question \ref{concavity} arises. In \cite{EGLX}, the authors conjectured that the answer to Question \ref{concavity} is positive.

When $n=2$, the conjecture is completely proved by Bergweiler-Eremenko-Langley \cite{BEL} via tools from complex analysis and geometric function theory. Actually in two dimensions, they obtain that all quasi-concave solutions are all one-dimensional. When $n \ge 3$, the conjecture is also validated in \cite{EGLX} by adding the following additional assumption:
\begin{align}
    \label{supcondition}
\mbox{$u$ takes its supremum at some point in  $\mathbb{R}^n$}.
\end{align} Indeed, \eqref{supcondition} and asymptotic analysis guarantees that a concave solution $u$ cannot be strictly concave everywhere, and thus by the seminal constant rank theorem proved in Korevaar-Lewis \cite{KL87}, the rank of its hessian must always be less than $n$, and then a dimension reduction argument can prove that $u$ depends only on one variable. For the details, we refer to \cite{EGLX}. Without the additional assumption \eqref{supcondition}, Question \ref{concavity} remains open for $n \ge 3$.

We remark that to solve Question \ref{concavity} in higher dimensions, the argument in the study of De Giorgi conjecture does not work. The De Giorgi conjecture is closely related to minimal surface theory and thus the tools there can produce essential progress in solving the De Giorgi conjecture. Due to Ghoussoub-Gui \cite{GG98} and Ambrosio-Cabr\'e \cite{AC20}, De Giorgi's conjecture is proved in dimension 2 and 3. In dimension 4 to 8, under a further limit condition, Savin \cite{Savin} validated the conjecture, and we also refer to Wang \cite{Wang2017} where a different proof is given. When $n \ge 9$, a counterexample is given by Del Pino-Kowalczyk-Wei \cite{DKW}. On the other hand, so far there has no clear geometric interpretation of the higher dimensional Liouville equation, though the toda system does describe interactions of finite ends of solutions to the Allen-Cahn equation, see the deep work of Wang-Wei \cite{wangwei}. Also, the non sign-changing property of directional derivative of a solution to the \eqref{acequation} implies that the solution is stable, which plays a crucial role in proving De Giorgi conjecture, while the concavity of solutions to \eqref{liouvilleequation} does not imply local minimality property in calculus of variations.

\medskip

Question \ref{concavity} motivates us to have a closer study of the constant rank theorem for elliptic equations,  which we believe is the right tool to solve Question \ref{concavity}, as illustrated in \cite{EGLX} under the extra assumption \eqref{supcondition}. Hence we turn our study to analyzing structures of convex solutions to the following more general semilinear elliptic equation
\begin{align}
\label{semi}
    \Delta u= G(u) \quad \mbox{in $\mathbb{R}^n$}
\end{align}where $G$ is a smooth function satisfying \begin{align}
    \label{conditonofG}
    G>0, \quad G'<0\quad \mbox{and}\quad G G'' \le A(G')^2,
\end{align}with $A$ being a positive constant. The equation \eqref{semi} with the condition \eqref{conditonofG} has been first considered in \cite{CF} and \cite{KL87}. Note that if $u$ is a concave solution to the Liouville equation \eqref{liouvilleequation}, then by letting $v=-u$, $v$ becomes a convex solution to \eqref{semi} with $G(v)=e^{-2v}$. Hence such $G$ satisfies \eqref{conditonofG}, and $A$ can be chosen to be $1$.

Constant rank theorem, though has its various versions, generally says that the hessian of a convex solution must have constant rank in its domain, if the equation satisfies certain structural conditions. The earliest version of constant rank theorem was due to Caffarelli-Friedman \cite{CF} and considers semilinear equations in $\mathbb{R}^2$, and was later generalized to $n$ dimensions by Korevaar-Lewis \cite{KL87}. We also refer to Bian-Guan \cite{BG}, Bian-Guan-Ma-Xu \cite{BGMX}, Bryan-Ivaki-Scheuer \cite{BIS23}, Caffarelli-Guan-Ma \cite{CGM}, Guan-Lin-Ma \cite{GLM}, Guan-Ma \cite{GM}, Guan-Ma-Zhou \cite{GMZ}, Guan-Xu \cite{GX}, Ma-Xu \cite{MX08}, Sun-Xu \cite{SX21} and  Sz$\acute{e}$kelyhidi-Weinkove\cite{SW16,SW21} for the case of fully nonlinear elliptic equations together with geometrical applications.

Constant rank theorem is usually served as a tool to prove (strict) convexity property of solutions defined in convex domains, via the argument of continuity method. For example in \cite{CF}, it is used to show that the critical point of Robin functions on convex domains must be non-degenerate. In \cite{EGLX}, it was the first time that the constant rank theorem was successfully applied to prove one-dimensional property of concave solutions, though the additional assumption \eqref{supcondition} is required. In order to drop the assumption \eqref{supcondition}, we need to have a better understanding of properties of convex solutions to \eqref{semi}, and this is the aim of the present paper. We will prove a stronger version of constant rank theorem.

\medskip

\subsection{Main Results}
By revisiting Caffarelli and Friedman's proof in \cite{CF}, we first prove the following result without much difficulty:
\begin{theorem}
\label{nostrictconvexitylemma}
Let $u$ be a convex solution to \eqref{semi} in $\mathbb{R}^2$, with $G$ satisfying \eqref{conditonofG}, where $A \le 2$. If $det (D^2 u)$ has a local minimum, then $det (D^2 u)=0$ everywhere.
\end{theorem}

Note that this theorem generalizes that in \cite{CF}, since for a convex solution $u$ to \eqref{semi}, $det(D^2 u)=0$ at some point $p$ automatically implies that $p$ is a local minimum point of $det(D^2 u)$. Our theorem says that the converse is also true.

Now we let $\sigma_k$ be the $k$-th elementary symmetric polynomials in $n$ variables, and hence when $n=2$, the determinant of a square matrix is the same as the $\sigma_2$ of the eigenvalues of the matrix. Surprisingly, it turns out that similar result to Theorem \ref{nostrictconvexitylemma} holds for $\sigma_2(D^2 u)$ in higher dimensions, while the proof is yet much more complicated.

\begin{theorem}
\label{sigma2theorem}
Let $u$ be a convex solution to the semilinear elliptic equation \eqref{semi} in $\Omega \subset \mathbb{R}^n$, with $G$ satisfying \eqref{conditonofG}, where $A \le 2$. If $\sigma_2(D^2 u)$ has a local minimum in the interior, the $D^2u$ must have rank $1$ everywhere in its domain.
\end{theorem}

\begin{remark*}
$A=2$ is the best constant in order that the argument of our proof works. Also, the last inequality in \eqref{conditonofG} is equivalent to saying that $1/G$ is convex.
\end{remark*}

Furthermore, We give another different generalization of Theorem \ref{nostrictconvexitylemma} in $n$ dimensions, by considering the case of $\sigma_n$.
\begin{theorem}
\label{sigman}
Let $u$ be a convex solution to \eqref{semi} in $\Omega \subset \mathbb{R}^n$, with $G>0, G'<0$ and  $GG^{''}\le \frac{n}{n-1}(G')^2$ in $\Omega$. Then if $det(D^2 u)$ has a local minimum, then $det(D^2 u)$ is always equal to $0$.
\end{theorem}
Again, the proof of Theorem \ref{sigman} requires very sophisticated computation and sign analysis of plenty of algebraic terms and matrices, and the constant $\frac{n}{n-1}$ seems to be optimal. At this moment we do not know if similar conclusions hold for $\sigma_k(D^2 u),\, 3\le k\le n-1$, when $n \ge 4$.

Clearly Theorem \ref{sigma2theorem} and Theorem \ref{sigman} cover Theorem \ref{nostrictconvexitylemma}, but we still would like to include a separate proof of Theorem \ref{nostrictconvexitylemma}, since the proof is very simple and directly follows from previous computation in \cite{CF}.

\medskip

A possible application of Theorem \ref{sigma2theorem} is that if we can show the $\sigma_2$ of the hessian of any concave solution to \eqref{liouvilleequation} has a local minimum, then by Theorem \ref{sigma2theorem} and the result in \cite{KL87}, concave solutions are one-dimensional. Also, if we can show that the determinant of the hessian of any concave solution to  \eqref{liouvilleequation} has a local minimum, then by Theorem \ref{sigman} and a dimension reduction argument, we can still conclude that concave solutions are one-dimensional.

Even though at this moment we still cannot reach the two possible directions, we believe our Theorem \ref{sigma2theorem} and Theorem \ref{sigman} are still interesting by themselves, since they can be viewed as stronger constant rank theorems in the case of $\sigma_2$ and $\sigma_n$, and they might have applications in other contexts. Also, Theorem \ref{sigma2theorem} and Theorem \ref{sigman} somehow suggest that the answer to Question \ref{concavity} is positive.

\medskip

The outline of the paper is as follows. In section 2, we prove Theorem \ref{nostrictconvexitylemma}. In section 3, we prove Theorem \ref{sigma2theorem}. In section 4, we prove Theorem \ref{sigman}.

\section{Proof of Theorem \ref{nostrictconvexitylemma}}
\begin{proof}[Proof of Theorem \ref{nostrictconvexitylemma}]
Let $\phi =2 det (D^2u)=(\Delta u)^2 -|D^2 u|^2 =G^2-|D^2 u|^2$, where $|D^2 u|^2 =\sum_i \sum_j u_{ij}^2$.  Hence for $k=1,2$, \begin{align}
\label{gphi_k}
    \phi_k=2GG'u_k-2\sum_{i,j}u_{ij}u_{ijk}
\end{align} and
\begin{align}
\label{erweishizi}
    \sum_k\phi_{kk}=&2(G')^2|\nabla u|^2 +2GG^{''}|\nabla u|^2+2GG'\Delta u-2|\nabla^3 u|^2 -2\sum_{i,j}u_{ij}(\Delta u)_{ij}\nonumber\\
    =& \left(2(G')^2+2GG''\right)|\nabla u|^2 +2 G^2G'-2 |\nabla ^3 u|^2-2\sum_{i,j}u_{ij}G'u_{ij}-2\sum_{i,j}u_{ij}G''u_iu_j\nonumber\\
   = & \left(2(G')^2+2GG''\right)|\nabla u|^2 +2 G'\phi-2 |\nabla ^3 u|^2-2\sum_{i,j}u_{ij}G''u_iu_j.
\end{align}
Suppose that $\phi$ takes a local minimum at $p$. If $\phi(p)>0$, then we may WLOG assume that \begin{align}
    \label{grotation}u_{ij}(p)=0, \, i\ne j.
\end{align} At $p$, we have for each $k=1,2$,
\begin{align}
\label{gfirstvanishing}
    \phi_k=0.
\end{align}
Since $G'(u)u_k=(-\Delta u)_k$, by \eqref{gphi_k}-\eqref{gfirstvanishing} we have
\begin{align*}
    \sum_{i,j}u_{iik}u_{jj}-\sum_iu_{ii}u_{iik}=0,
\end{align*}that is,
\begin{align}
\label{gfirstobservation}
    u_{11}u_{22i}+u_{22}u_{11i}=0,\quad i=1,2.
\end{align}
Differentiating the equation \eqref{semi}, we have
\begin{align}
    \label{gsecondobservation}
u_{11i}+u_{22i}=G' u_i,\quad i=1,2.
\end{align}
If at $p$, $u_{11} \ne u_{22}$, combining \eqref{gfirstobservation} and \eqref{gsecondobservation}, for each $i=1,2$, we have
\begin{align}
    \label{g3rdderivativerepresentation}
u_{11i}=\frac{G'u_{11}u_i}{u_{11}-u_{22}} \quad \mbox{and}\quad u_{22i}=\frac{G'u_{22}u_i}{u_{22}-u_{11}}.
\end{align}
Since $\phi(p)=2u_{11}(p)u_{22}(p)>0$, we may assume $u_{11}(p)>u_{22}(p)>0$. Set $q:=\frac{u_{22}(p)}{u_{11}(p)-u_{22}(p)}> 0$. Since $u_{11}+u_{22}=G$, at $p$ we have
\begin{align*}
    u_{11}=\frac{1+q}{2q+1}G \quad \mbox{and}\quad u_{22}=\frac{q}{2q+1}G.
\end{align*}
Hence at $p$, \begin{align*}
    \frac{u_{11}}{u_{11}-u_{22}}=1+q.
\end{align*}
Therefore, by \eqref{g3rdderivativerepresentation}, we have that \begin{align}
    \label{nabla3}
|\nabla ^3 u|^2=u_{111}^2+3u_{221}^2+u_{222}^2+3u_{112}^2=\left((1+q)^2+3q^2\right)(G')^2|\nabla u|^2.
\end{align}
Therefore, from \eqref{erweishizi} we have
\begin{align*}
    \Delta \phi-2G'\phi=&2u_1^2\left((G')^2+GG''-(G')^2\left((1+q)^2+3q^2\right)-\frac{1+q}{1+2q}G''G\right)\\
    &+2u_2^2\left((G')^2+GG''-(G')^2\left(q^2+3(1+q)^2\right)-\frac{q}{1+2q}G''G\right)\\
    =&2u_1^2\left((G')^2(-4q^2-2q)+\frac{q}{1+2q}G''G\right)\\
    &+2u_2^2\left((G')^2(-4q^2-6q-2)+\frac{1+q}{1+2q}G''G\right)\\
    \le &2u_1^2(G')^2\left(-4q^2-2q+\frac{Aq}{1+2q}\right)\\
    &+2u_2^2(G')^2\left(-4q^2-6q-2+A\frac{1+q}{1+2q}\right).
\end{align*}
Let $B_1(q)=-4q^2-2q+\frac{Aq}{1+2q}$ and $B_2(q)=-4q^2-6q-2+A\frac{1+q}{1+2q}$. One can check $A=2$ is the best constant such that both $B_1(q)\le 0$ and $B_2(q) \le 0$ for all $q \ge 0$. Hence at $p$, $\Delta \phi -2G' \phi \le 0$, which contradicts the assumption that $\phi(p)$ is takes a local positive minimum value.

If at $p$ we have $u_{11}=u_{22}$, then $u_{11}=u_{22}=G/2$, and by
\eqref{gfirstobservation} and \eqref{gsecondobservation} we have that $\nabla u=0$. Hence at $p$m
\begin{align*}
    \Delta \phi-2G \phi=-2|\nabla ^3 u|^2 \le 0.
\end{align*}This still leads to a contradiction.
Hence if a local minimum of $\phi$ is achieved at $p$, $\phi(p)=0$ and thus $\phi \equiv 0$ by \cite{CF}.

\end{proof}

\section{Compuations in the case of $\sigma_2$}
In this section, we prove Theorem \ref{sigma2theorem}.
\begin{proof}[Proof of Theorem \ref{sigma2theorem}]
Let $\phi=2\sigma_2=(\Delta u)^2-\sum_{i,j}u_{ij}^2=\sum_{1\le i\ne j\le n}u_{ii}u_{jj}-\sum_{1\le i\ne j\le n}u_{ij}^2$. Suppose that $\phi$ takes a local minimum at $p$. If $\phi(p)=0$, then by constant rank theorem proved in \cite{KL87}, the conclusion automatically holds.

From now on, we assume that $\phi(p)>0$. Taking the derivative with respect to the $k$-th variable, we have
\begin{align}
\label{hphik}
    \phi_k=2\sum_{1\le i\ne j\le n}u_{ii}u_{jjk}-2\sum_{1\le i\ne j\le n}u_{ij}u_{ijk}.
\end{align}
At the point $p$, we may assume that $u_{ij}(p)$ is diagonalized. Also, clearly at least two of $u_{ii}(p), \, 1\le i\le n$ do not vanish, otherwise $\phi(p)$ would be zero.

Taking the $k$-th derivative on \eqref{hphik}, summing over $k$, and using that $u_{ij}(p)=0$ when $i\ne j$, we have that at $p$,
\begin{align}
\label{hphikk}
    \frac{1}{2}\Delta \phi=\sum_k\sum_{1\le i\ne j\le n}\left(u_{iik}u_{jjk}+u_{ii}u_{jjkk}-u_{ijk}^2\right):=I+II+III.
\end{align}
Now we first estimate
\begin{align}
\label{II1}
II=&\sum_{1\le i\ne j\le n}u_{ii}(\Delta u)_{jj}\nonumber\\
=&\sum_{i=1}^nu_{ii}\sum_{j\ne i}(G'u_{jj}+G^{''}u_j^2)\nonumber\\
=&G'\phi+G^{''}\sum_{i=1}^nu_{ii}\sum_{j\ne i}u_j^2\nonumber\\
= &G'\phi+\frac{G^{''}G\sum_{i=1}^nu_{ii}\sum_{j\ne i}(G'u_j)^2}{G(G')^2}\nonumber\\
\le&G'\phi+\frac{A}{G}\sum_{i=1}^nu_{ii}\sum_{j\ne i}(G'u_j)^2,
\end{align}
where we have used  \eqref{conditonofG}. Then at $p$, since $\phi_k=0$, from \eqref{hphik} we have
\begin{align*}
    \sum_{i=1}^nu_{ii}\sum_{j \ne i}u_{jjk}=0.
\end{align*}
Rewriting the above, for any fixed $k$ we have
\begin{align}
\label{xu1}
    \sum_{i\ne k}u_{ii}u_{kkk}+\sum_{j\ne k}\sum_{i \ne j}u_{ii}u_{jjk}=0.
\end{align}
Taking the $k$-th derivative on both sides of the equation \eqref{semi}, we have
\begin{align}
\label{xu2}
    u_{kkk}+\sum_{j\ne k}u_{jjk}=G'u_k.
\end{align}
Hence by \eqref{xu1} and \eqref{xu2}, we have
\begin{align}
\label{xu3}
    G'u_k=&-\frac{\sum_{j\ne k}\sum_{i \ne j}u_{ii}u_{jjk}}{\sum_{i\ne k}u_{ii}}+\sum_{j\ne k}u_{jjk}\nonumber\\
    =&\sum_{j\ne k}u_{jjk}\left(1-\frac{\sum_{i \ne j}u_{ii}}{\sum_{i\ne k}u_{ii}}\right)\nonumber\\
    =&\sum_{j\ne k}u_{jjk}\frac{u_{jj}-u_{kk}}{\sum_{i\ne k}u_{ii}}.
\end{align}
Plugging \eqref{xu3} into \eqref{II1}, we have
\begin{align}
    \label{II2}
II\le& G'\phi+ \frac{A}{G}\sum_{i=1}^nu_{ii}\sum_{j\ne i}\left(\sum_{l\ne j}u_{llj}\frac{u_{ll}-u_{jj}}{\sum_{m\ne j}u_{mm}}\right)^2\nonumber\\
=&G'\phi+ \frac{A}{G}\sum_{j=1}^n\sum_{i\ne j}u_{ii}\left(\sum_{l\ne j}u_{llj}\frac{u_{ll}-u_{jj}}{\sum_{m\ne j}u_{mm}}\right)^2\nonumber\\
=&G'\phi+ \frac{A}{G}\sum_{j=1}^n(\Delta u-u_{jj})\left(\sum_{l\ne j}u_{llj}\frac{u_{ll}-u_{jj}}{\Delta u-u_{jj}}\right)^2\nonumber\\
=&G'\phi+ \frac{A}{G}\sum_{j=1}^n\frac{1}{\Delta u-u_{jj}}\sum_{l\ne j}\left(u_{llj}^2(u_{ll}-u_{jj})^2+\sum_{k\ne l,j}u_{llj}u_{kkj}(u_{ll}-u_{jj})(u_{kk}-u_{jj})\right).
\end{align}
Next, we rewrite $I+III$ defined in \eqref{hphikk} as
\begin{align}
\label{1+3}
    I+III=&2\sum_{1\le i\ne j\le n}u_{iii}u_{jji}+\sum_{1\le i\ne j\le n}\sum_{k\ne i,j}u_{iik}u_{jjk}-2\sum_{1\le i\ne j\le n}u_{ijj}^2-\sum_{1\le i\ne j\le n}\sum_{k\ne i,j}u_{ijk}^2\nonumber\\
    \le&2\sum_{1\le i\ne j\le n}u_{iii}u_{jji}+\sum_{1\le i\ne j\le n}\sum_{k\ne i,j}u_{iik}u_{jjk}-2\sum_{1\le i\ne j\le n}u_{ijj}^2.
\end{align}
By \eqref{xu1} and \eqref{1+3}, we have
\begin{align}
\label{13}
    I+III\le& -2\sum_{1\le i\ne j\le n}\frac{\sum_{l\ne i}\sum_{m \ne l}u_{mm}u_{lli}}{\sum_{m\ne i}u_{mm}}u_{jji}+\sum_{1\le i\ne j\le n}\sum_{k\ne i,j}u_{iik}u_{jjk}-2\sum_{1\le i\ne j\le n}u_{ijj}^2\nonumber\\
    =&-2\sum_{1\le i\ne j\le n}\frac{\sum_{l\ne i}(\Delta u-u_{ll})u_{lli}}{\Delta u-u_{ii}}u_{jji}+\sum_{1\le i\ne j\le n}\sum_{k\ne i,j}u_{iik}u_{jjk}-2\sum_{1\le i\ne j\le n}u_{ijj}^2\nonumber\\
    =&-2\sum_{1\le i\ne j\le n}\left(\frac{\Delta u-u_{jj}}{\Delta u-u_{ii}}+1\right)u_{ijj}^2-2\sum_{1\le i\ne j\le n}\sum_{l \ne i,j}\frac{\Delta u-u_{ll}}{\Delta u-u_{ii}}u_{lli}u_{jji}+\sum_{1\le i\ne j\le n}\sum_{k\ne i,j}u_{iik}u_{jjk}\nonumber\\
    =&-2\sum_{1\le i\ne j\le n}\left(\frac{\Delta u-u_{jj}}{\Delta u-u_{ii}}+1\right)u_{ijj}^2+\sum_{1\le i\ne j\le n}\sum_{l \ne i,j}\left(1-2\frac{\Delta u-u_{ll}}{\Delta u-u_{ii}}\right)u_{lli}u_{jji}\nonumber\\
    =&-2\sum_{1\le i\ne j\le n}\frac{2\Delta u-u_{ii}-u_{jj}}{\Delta u-u_{ii}}u_{ijj}^2+\sum_{1\le i\ne j\le n}\sum_{l \ne i,j}\frac{2u_{ll}-\Delta u-u_{ii}}{\Delta u-u_{ii}}u_{lli}u_{jji}.
\end{align}
We also rewrite \eqref{II2} as
\begin{align}
\label{II3}
    II \le G'\phi+\frac{A}{G}\sum_{1\le i \ne j \le n}\frac{1}{\Delta u-u_{ii}}\left(u_{jji}^2(u_{jj}-u_{ii})^2+\sum_{l\ne i,j}u_{jji}u_{lli}(u_{jj}-u_{ii})(u_{ll}-u_{ii})\right).
\end{align}
Now by \eqref{hphikk}, \eqref{13} and \eqref{II3}, we obtain
\begin{align}
\label{xu4}
    \frac{1}{2}\Delta \phi-G'\phi \le& \sum_{i=1}^n\frac{1}{G(\Delta u-u_{ii})}\Big(\sum_{j\ne i}\left(-2G(2\Delta u-u_{ii}-u_{jj})+A(u_{jj}-u_{ii})^2\right)u_{jji}^2\nonumber\\
    &+\sum_{j\ne i}\sum_{l\ne i,j}\left(G(2u_{ll}-\Delta u-u_{ii})+A(u_{jj}-u_{ii})(u_{ll}-u_{ii})\right)u_{lli}u_{jji}\Big)\nonumber\\
    =&\sum_{i=1}^n\frac{1}{G(\Delta u-u_{ii})}\Big(\sum_{j\ne i}\left(-2G(2\Delta u-u_{ii}-u_{jj})+A(u_{jj}-u_{ii})^2\right)u_{jji}^2\nonumber\\
    &+\sum_{j\ne i}\sum_{l\ne i,j}\left(G(u_{jj}+u_{ll}-\Delta u-u_{ii})+A(u_{jj}-u_{ii})(u_{ll}-u_{ii})\right)u_{lli}u_{jji}\Big).
\end{align}
Concerning the coefficients in front of $u_{jji}^2$, note that
\begin{align}
    \label{xu5}
&-2G(2\Delta u-u_{ii}-u_{jj})+A(u_{jj}-u_{ii})^2\nonumber\\
=&-2\Delta u(2\sum_{k\ne i,j}u_{kk}+u_{ii}+u_{jj})+ A(u_{jj}-u_{ii})^2 \nonumber \\
=& -2(\sum_{k\ne i,j}u_{kk}+u_{ii}+u_{jj})(2\sum_{k\ne i,j}u_{kk}+u_{ii}+u_{jj})+ A(u_{jj}-u_{ii})^2\nonumber\\
=&-4(\sum_{k\ne i,j}u_{kk})^2-6\sum_{k\ne i,j}u_{kk}(u_{ii}+u_{jj})+(A-2)(u_{ii}-u_{jj})^2 -8u_{ii}u_{jj}.
\end{align}
Hence in order this to be negative, the best choice is that $A \le 2$.  Hence from \eqref{II1}, and previous computations, we may let $A=2$ in \eqref{xu4}. Concerning the coefficients in front of mixed third order derivatives, we have that
\begin{align}
    \label{xu6}
    &G(u_{jj}+u_{ll}-\Delta u-u_{ii})+2(u_{jj}-u_{ii})(u_{ll}-u_{ii})\nonumber\\
    =&-\Delta u(\Delta u+u_{ii}-u_{jj}-u_{ll})+2(u_{jj}-u_{ii})(u_{ll}-u_{ii})\nonumber\\
    =&-(\sum_{k\ne i,j,l}u_{kk}+u_{ii}+u_{jj}+u_{ll})(2u_{ii}+\sum_{k\ne i,j,l}u_{kk})+2(u_{jj}-u_{ii})(u_{ll}-u_{ii})\nonumber\\
    =&-(\sum_{k\ne i,j,l}u_{kk})^2-(\sum_{k\ne i,j,l}u_{kk})(3u_{ii}+u_{jj}+u_{ll})-4u_{ii}(u_{ll}+u_{jj})+2u_{jj}u_{ll}.
\end{align}
Hence by \eqref{xu4}-\eqref{xu6}, we have
\begin{align}
\label{xu7}
    \frac{1}{2}\Delta \phi-G'\phi \le& \sum_{i=1}^n\frac{1}{G(\Delta u-u_{ii})}\Bigg[-\sum_{j\ne i}\Big(4(\sum_{k\ne i,j}u_{kk})^2+6\sum_{k\ne i,j}u_{kk}(u_{ii}+u_{jj})+8u_{ii}u_{jj}\Big)u_{jji}^2\nonumber\\
    &+\sum_{j\ne i}\sum_{l\ne i,j}\Big(-(\sum_{k\ne i,j,l}u_{kk})^2-(\sum_{k\ne i,j,l}u_{kk})(3u_{ii}+u_{jj}+u_{ll})-4u_{ii}(u_{ll}+u_{jj})+2u_{jj}u_{ll}\Big)u_{lli}u_{jji}\Bigg]\nonumber\\
    =&\sum_{i=1}^n\frac{1}{G(\Delta u-u_{ii})} (I'+II'+III'),
\end{align}
where
\begin{align*}
    I'=-\sum_{j\ne i}\sum_{k\ne i,j}4u_{kk}^2u_{jji}^2-\sum_{j\ne i}\sum_{l\ne i,j}\sum_{k\ne i,j,l}u_{kk}^2u_{lli}u_{jji},
\end{align*}
\begin{align*}
    II'=&-\sum_{j\ne i}\left(4\sum_{k\ne i,j}\sum_{l\ne i,j,k}4u_{kk}u_{ll}
+6\sum_{k\ne i,j}u_{kk}u_{jj}\right)u_{jji}^2\nonumber\\
&-\sum_{j\ne i}\sum_{l\ne i,j}\Big(\sum_{k\ne i,j,l}\sum_{m\ne i,j,k,l}u_{kk}u_{mm}+(\sum_{k\ne i,j,l}u_{kk})(u_{jj}+u_{ll})-2u_{jj}u_{ll}\Big)u_{lli}u_{jji}\nonumber\\
\end{align*}and
\begin{align*}
    III'=-u_{ii}\sum_{j\ne i}\left((6\sum_{k\ne i,j}u_{kk}+8u_{jj})u_{jji}^2+\sum_{l\ne i,j}(3\sum_{k\ne i,j,l}u_{kk}+4u_{ll}+4u_{jj})u_{lli}u_{jji}  \right).
\end{align*}
We rewrite $I', II', III'$ as the following.
\begin{align}
\label{I1}
    I'    =&-\sum_{k\ne i}u_{kk}^2\left(4\sum_{j\ne i,k}u_{jji}^2+\sum_{j\ne i,k}\sum_{l\ne i,j,k}u_{jji}u_{lli}\right)\nonumber\\
    =&-\sum_{k\ne i}u_{kk}^2\left((\sum_{j\ne i,k}u_{jji})^2+3\sum_{j\ne i,k}u_{jji}^2\right).
\end{align}
\begin{align}
    \label{I2}
II'=&-\sum_{k\ne i}\sum_{j\ne i,k}u_{jj}u_{kk}\left(4\sum_{l\ne i,j,k}u_{lli}^2+6u_{jji}^2+\sum_{l\ne i,j,k}\sum_{m \ne i,j,k,l}u_{lli}u_{mmi}+2u_{jji}\sum_{l\ne i,j,k}u_{lli}-2u_{jji}u_{kki}\right)\nonumber\\
=&-\sum_{k\ne i}\sum_{j\ne i,k}u_{jj}u_{kk}\left(3\sum_{l\ne i,j,k}u_{lli}^2+5u_{jji}^2+u_{kki}^2+(\sum_{l\ne i,j,k}u_{lli})^2+2u_{jji}\sum_{l\ne i,j,k}u_{lli}-2u_{jji}u_{kki}\right)\nonumber\\
=&-\sum_{k\ne i}\sum_{j\ne i,k}\left(3\sum_{l\ne i,j,k}u_{lli}^2+(u_{jji}+\sum_{l\ne i,j,k}u_{lli})^2+(u_{jji}-u_{kki})^2+3u_{jji}^2\right).
\end{align}
\begin{align}
    \label{I3}
III'=&-u_{ii}\sum_{j\ne i}u_{jj}\left(6\sum_{k\ne i,j}u_{kki}^2+8u_{jji}^2+3\sum_{k\ne i,j}\sum_{l\ne i,j,k}u_{lli}u_{kki}+8u_{jji}\sum_{l\ne i,j}u_{lli}  \right)\nonumber\\
=&-u_{ii}\sum_{j\ne i}u_{jj}\left(3(\frac{4}{3}u_{jji}+\sum_{k\ne i,j}u_{kki})^2+\frac{8}{3}u_{jji}^2+3\sum_{k\ne i,j}u_{kki}^2 \right).
\end{align}
Hence by \eqref{xu7}-\eqref{I3}, we have shown that at $p$,
\begin{align*}
    \frac{1}{2}\Delta \phi-G'\phi \le 0.
\end{align*}
This contradicts to the assumption that $\phi$ takes a local positive minimum at $p$. Hence $\phi(p)=0$ and thus again by constant rank theorem, $D^2 u$ has rank $1$ everywhere.
\end{proof}

\section{Computation in the case of $\sigma_k$ and proof of Theorem \ref{sigman}}
\begin{proof}[Proof of Theorem \ref{sigman}]
Let $GG^{''}\le A(G')^2$, $\phi=\sigma_k(D^2 u)$, and we suppose that $\phi$ takes a local minimum at $p$. We may assume that $D^2 u$ is diagonalized at $p$ up to changing coordinates. In the following, we denote by $\sigma_k(\lambda|i)$ the symmetric function with $D^2 u$ deleting the $i$-row and $i$-column, and by $\sigma_k(\lambda|i, j)$ the symmetric function with $D^2 u$ deleting the $i, j$-rows and $i, j$-columns.

Then for $\forall \;1 \le m \le n$, it holds that at $p$,
\begin{align*}
    0=\phi_m=\sum_{i,j}\frac{\partial \sigma_k}{\partial u_{ij}}u_{ijm}=\sum_i\sigma_{k-1}(\lambda|i)u_{iim}.
\end{align*}
Hence at $p$, if $\sigma_{k-1}(\lambda|m)\ne 0$ for every $m\in \{1, \cdots, n\}$ (In fact, it is satisfied in our case when $k=n$, see the explanation after \eqref{finaln}), then we have
\begin{align}
\label{ummn}
    u_{mmm}=-\frac{\sum_{i\ne m}\sigma_{k-1}(\lambda|i)u_{iim}}{\sigma_{k-1}(\lambda|m)}.
\end{align}
Using \eqref{ummn} and equation \eqref{semi}, we compute that at $p$,
\begin{align}
    \Delta \phi=&\sum_m(\sum_{s,t,i,j}\frac{\partial^2 \sigma_k}{\partial u_{st}\partial u_{ij}}u_{stm}u_{ijm}+\sum_{i,j}\frac{\partial \sigma_k}{\partial u_{ij}}u_{ijmm})\nonumber\\
    =&\sum_m\sum_{i\ne j}\sigma_{k-2}(\lambda |i,j)u_{jjm}u_{iim}-\sum_m\sum_{i\ne j}\sigma_{k-2}(\lambda|i,j)u_{ijm}^2+\sum_{i,m}\sigma_{k-1}(\lambda|i)u_{iimm}\nonumber\\
    =&-2\sum_{m\ne j}\sigma_{k-2}(\lambda|m,j)u_{jjm}\frac{\sum_{l\ne m}\sigma_{k-1}(\lambda|l)u_{llm}}{\sigma_{k-1}(\lambda|m)}+\sum_{i,j \ne m}\sum_{i\ne j}\sigma_{k-2}(\lambda|i,j)u_{iim}u_{jjm}\nonumber\\
    &-2\sum_{i\ne m}\sigma_{k-2}(\lambda|i,m)u_{imm}^2-\sum_{i\ne j}\sum_{m\ne i,j}\sigma_{k-2}(\lambda|i,j)u_{ijm}^2+\sum_i\sigma_{k-1}(\lambda|i)(G^{''}u_i^2+G'u_{ii}).
\end{align}
Using that $$\sum_i u_{ii}\sigma_{k-1}(\lambda|i)=k\sigma_k=k\phi,$$
and \eqref{conditonofG}, we thus have
\begin{align}
\label{pie0}
    \Delta \phi-k G'\phi \le& -2\sum_{m\ne j}\sigma_{k-2}(\lambda|m,j)u_{jjm}\frac{\sum_{l\ne m}\sigma_{k-1}(\lambda|l)u_{llm}}{\sigma_{k-1}(\lambda|m)}+\sum_{i,j \ne m}\sum_{i\ne j}\sigma_{k-2}(\lambda|i,j)u_{iim}u_{jjm}\nonumber\\
    &-2\sum_{i\ne m}\sigma_{k-2}(\lambda|i,m)u_{iim}^2+\sum_i\frac{A}{G}(G'u_i)^2\sigma_{k-1}(\lambda|i)\nonumber\\
    =&-2\sum_{l\ne  m, j}\sum_{m\ne j}\sigma_{k-2}(\lambda|m,j)u_{llm}u_{jjm}\frac{\sigma_{k-1}(\lambda|l)}{\sigma_{k-1}(\lambda|m)}-2\sum_{j\ne  m}\sigma_{k-2}(\lambda|m,j)u_{jjm}^2\frac{\sigma_{k-1}(\lambda|j)}{\sigma_{k-1}(\lambda|m)}\nonumber\\
    &+\sum_{i,j \ne m}\sum_{i\ne j}\sigma_{k-2}(\lambda|i,j)u_{iim}u_{jjm}-2\sum_{i\ne m}\sigma_{k-2}(\lambda|i,m)u_{iim}^2+\sum_{i}\frac{A}{G}(\sum_{l}u_{ill})^2\sigma_{k-1}(\lambda|i)\nonumber\\
     =&-2\sum_{i\ne  m, j}\sum_{m\ne j}\sigma_{k-2}(\lambda|m,j)u_{iim}u_{jjm}\frac{\sigma_{k-1}(\lambda|i)}{\sigma_{k-1}(\lambda|m)}-2\sum_{i\ne  m}\sigma_{k-2}(\lambda|m,i)u_{iim}^2\frac{\sigma_{k-1}(\lambda|i)}{\sigma_{k-1}(\lambda|m)}\nonumber\\
    &+\sum_{i,j \ne m}\sum_{i\ne j}\sigma_{k-2}(\lambda|i,j)u_{iim}u_{jjm}-2\sum_{i\ne m}\sigma_{k-2}(\lambda|i,m)u_{iim}^2+\sum_m\frac{A}{G}(\sum_{i}u_{mii})^2\sigma_{k-1}(\lambda|m).
\end{align}
Note that by \eqref{ummn}, we get
\begin{align}
    \label{pie}
(\sum_{i}u_{mii})^2\sigma_{k-1}(\lambda|m)=&(u_{mmm}+\sum_{i\ne m}u_{iim})^2 \sigma_{k-1}(\lambda|m)\nonumber\\
=&\left(\sum_{i\ne m}\left(1-\frac{\sigma_{k-1}(\lambda|i)}{\sigma_{k-1}(\lambda|m)}\right)u_{iim}\right)^2\sigma_{k-1}(\lambda|m)\nonumber\\
=&\frac{\Big(\sum_{i\ne m}\sigma_{k-2}(\lambda|m,i)(\lambda_i-\lambda_m)u_{iim}\Big)^2}{\sigma_{k-1}(\lambda|m)}.
\end{align}
Hence \eqref{pie0}-\eqref{pie} lead to that at $p$,
\begin{align}
    \label{finaln}
&\Delta \phi-kG'\phi \nonumber\\ \le &\sum_{m=1}^n\frac{1}{G\sigma_{k-1}(\lambda|m)}\Bigg[\sum_{i\ne m}\Big(A\sigma_{k-2}^2(\lambda|m,i)(\lambda_i-\lambda_m)^2-2G\sigma_{k-2}(\lambda|m,i)(\sigma_{k-1}(\lambda|m)+\sigma_{k-1}(\lambda|i))\Big)u_{iim}^2\nonumber\\
&+\sum_{i,j \ne m}\sum_{i\ne j}\Big(A\sigma_{k-2}(\lambda|m,i)\sigma_{k-2}(\lambda|m,j)(\lambda_i-\lambda_m)(\lambda_j-\lambda_m)\nonumber\\
&-2G\sigma_{k-2}(\lambda|m,j)\sigma_{k-1}(\lambda|i)+G\sigma_{k-1}(\lambda|m)\sigma_{k-2}(\lambda|i,j)\Big)u_{iim}u_{jjm}\Bigg]\nonumber\\
:=&\sum_{m=1}^n\frac{1}{G\sigma_{k-1}(\lambda|m)} (I+II).
\end{align}

In the following we suppose $k=n$, and set $u_{ii}(p)=\lambda_i$. To prove Theorem \ref{sigman}, by \cite{KL87}, it suffices to show $\phi(p)=0$. If $\phi(p)>0$, then for each $m$, $\sigma_{n-1}(\lambda|m)\ne 0$ and the previous computations are valid since in every step the denominants are not zero.

Then for $\forall$ $m$ fixed,
\begin{align}
    \label{1n}
I=&\sum_{i\ne m}u_{iim}^2\Big(A\frac{\sigma_n^2}{\lambda_m^2\lambda_i^2}(\lambda_i-\lambda_m)^2-2G\frac{\sigma_n}{\lambda_i\lambda_m}(\frac{\sigma_n}{\lambda_m}+\frac{\sigma_n}{\lambda_i})\Big)    \nonumber\\
=&\sum_{i\ne m}u_{iim}^2\frac{\sigma_n^2}{\lambda_m^2\lambda_i^2}\Big((A-2)(\lambda_i^2+\lambda_m^2)-(2A+4)\lambda_i\lambda_m-2\sum_{l\ne i,m}\lambda_l(\lambda_i+\lambda_m)\Big),
\end{align}
and
\begin{align}
    \label{2n}
II=&  \sum_{i,j \ne m}\sum_{i\ne j}u_{iim}u_{jjm}\Big(A\frac{\sigma_n}{\lambda_i\lambda_m}\frac{\sigma_n}{\lambda_j\lambda_m}(\lambda_j-\lambda_m)(\lambda_i-\lambda_m)-G(2\frac{\sigma_n}{\lambda_m\lambda_j}\frac{\sigma_n}{\lambda_i}-\frac{\sigma_n}{\lambda_m}\frac{\sigma_n}{\lambda_i\lambda_j})\Big)\nonumber\\
=&\frac{\sigma_n^2}{\lambda_m^2}\sum_{i,j \ne m}\sum_{i\ne j}u_{iim}u_{jjm}\frac{1}{\lambda_i\lambda_j}\Big((A-1)\lambda_m^2+A\lambda_i\lambda_j-(A+1)\lambda_i\lambda_m-(A+1)\lambda_j\lambda_m-\sum_{l\ne i,j,m}\lambda_l\lambda_m\Big).
\end{align}
So we can rewrite
\begin{align}
    \label{1+2}
I+II=\frac{\sigma_n^2}{\lambda_m^2}\left( a_1\lambda_m^2+a_2\lambda_m+a_3\right),
\end{align}where
\begin{align}
    \label{a1}
a_1=\sum_{i\ne j}\sum_{i,j\ne m}(A-1)\frac{u_{iim}u_{jjm}}{\lambda_i\lambda_j}+\sum_{i\ne m}\frac{A-2}{\lambda_i^2}u_{iim}^2,
\end{align}
\begin{align}
\label{a2}
    a_2=&-\sum_{i,j\ne m}\sum_{i\ne j}\frac{2A+2}{\lambda_j}u_{iim}u_{jjm}-\sum_{i,j\ne m}\sum_{i\ne j}(\sum_{l\ne i,j,m}\lambda_l) \frac{u_{iim}u_{jjm}}{\lambda_i\lambda_j}\nonumber\\
    &-\sum_{i\ne m}\frac{2A+4}{\lambda_i}u_{iim}^2-\sum_{i\ne m}\frac{u_{iim}^2}{\lambda_i^2}(2\sum_{l\ne i,m}\lambda_l),
\end{align}
\begin{align}
    \label{a3}
a_3=&\sum_{i\ne j}\sum_{i,j\ne m} Au_{iim}u_{jjm}-2\sum_{i\ne m}u_{iim}^2\frac{\sum_{l\ne i,m}\lambda_l}{\lambda_i} +\sum_{i\ne m}(A-2)u_{iim}^2.
\end{align}
Let $A \in (1,2)$ and thus $\alpha: =\frac{2-A}{A-1} \in (0,\infty)$. Furthermore we define $a_{ij}=\alpha$ when $i=j$, and $a_{ij}=-1$ when $i\ne j$. Hence in order that $a_1\le 0$, it suffices that $(n-1) \times (n-1)$ matrix $(a_{ij})$ must be positive semi-definite. Let $e$ be the ones vector, so the $k$-th leading principle minor $L_k$ of the matrix $(a_{ij})$ is $(\alpha+1) I_k-e \otimes e$, where $I_k$ is the $k\times k$ identity matrix. We compute the determinant of $L_k$ as
\begin{align*}
    det (L_k)=(\alpha+1)^kdet(I_k-\frac{e\otimes e}{\alpha+1})=(\alpha+1)^k(1-\frac{e\cdot e}{\alpha+1})=(\alpha+1)^{k-1}(\alpha+1-k).
\end{align*}
Hence $(a_{ij})\ge 0$ is equivalent to $det(L_k)\ge 0$, and hence $\alpha \ge k-1$ for all $k=1,2, \cdots, (n-1)$. Hence in order that $a_1\le 0$, we need $\alpha=\frac{2-A}{A-1}\ge n-2$. Hence $A \in (1,\frac{n}{n-1}]$. Therefore, when $A\le \frac{n}{n-1}$, $a_1 \le 0$.

Let $y_i=\frac{u_{iim}}{\lambda_i}$, $\lambda=\sum_{i\ne m}\lambda_i$ and $p_i=\frac{\lambda_i}{\lambda}$. Then we have
\begin{align}
\label{a20}
    a_2=&-\lambda\left(\sum_{i,j\ne m}\sum_{i\ne j}(2A+2)p_iy_iy_j+\sum_{i,j\ne m}\sum_{i\ne j}(1-p_i-p_j)y_iy_j+\sum_{i\ne m}(2A+4)p_iy_i^2+\sum_{i\ne m}(2-2p_i)y_i^2\right)\nonumber\\
    =&-\lambda \left(\sum_{i\ne m}\Big(2+(2A+2)p_i\Big)y_i^2+\sum_{i,j\ne m}\sum_{i\ne j}(1+2Ap_i)y_iy_j\right)
\end{align}
Since
\begin{align*}
    \sum_{i,j\ne m}\sum_{i\ne j}p_iy_iy_j=(\sum_{i\ne m}p_iy_i)(\sum_{i\ne m}y_i)-\sum_{i\ne m}p_iy_i^2,
\end{align*} we have
\begin{align}
    \label{a21}
-a_2/\lambda=&\sum_{i,j\ne m}\sum_{i\ne j}y_iy_j+2A(\sum_{i\ne m}p_iy_i)(\sum_{i\ne m}y_i)+2\sum_{i\ne m}p_iy_i^2+2\sum_{i\ne m}y_i^2\nonumber\\
=&(\sum_{i\ne m}y_i+A\sum_{i\ne m}p_iy_i)^2+\sum_{i\ne m}y_i^2+2\sum_{i\ne m}p_iy_i^2-A^2(\sum_{i\ne m}p_iy_i)^2
\end{align}
Since $\sum_{i\ne m}p_i=1$, from Schwarz inequality we have
\begin{align*}
    (\sum_{i\ne m}p_iy_i)^2 \le \sum_{i\ne m}p_iy_i^2.
\end{align*}
Hence by \eqref{a21}, we have that when $A \le \sqrt{2}$, $a_2 \ge 0$.  When $n \ge 4$, $\frac{n}{n-1} <\sqrt{2}$, and hence $a_2 \le 0$ if $n\ge 4$.

When $n=3$ and $A=3/2$, one can still conclude that $a_2 \le 0$ by directly looking at \eqref{a20} and finding out that the matrix
\begin{align*}
    \left(\begin{matrix}
       2+5p_i & 1+3p_i\\
    1+3p_j & 2+5p_j
    \end{matrix}\right)
\end{align*}
is positive definite when $p_i+p_j=1$, $p_i, p_j \ge 0$.

Therefore, we have shown that when $A=\frac{n}{n-1}$, $a_2\le 0$, for all dimensions $n\ge 3$.

Now we estimate $a_3$. Let $b_i=u_{iim}$, and $\lambda=\sum_{i\ne m}\lambda_i$. Hence
\begin{align}
    \label{a31}
a_3=&-\left(\sum_{i\ne m}\sum_{j\ne i,m}\Big(\frac{2\lambda_j}{\lambda_i}b_i^2 -Ab_ib_j\Big)+(2-A)\sum_{i\ne m}b_i^2\right)\nonumber\\
=&-\left(\sum_{i\ne m}\sum_{j\ne i,m}\Big(\frac{\lambda_j}{\lambda_i}b_i^2+\frac{\lambda_i}{\lambda_j}b_j^2 -Ab_ib_j\Big)+(2-A)\sum_{i\ne m}b_i^2\right)
\end{align}
Hence for $A \le 2$, by \eqref{a31}, $a_3\le 0$. Hence  \eqref{a1}-\eqref{a3} are all nonpositive when $A\le \frac{n}{n-1}$.

Overall, by \eqref{finaln}-\eqref{a31}, at $p$ we have that when $A\le \frac{n}{n-1}$,
\begin{align*}
    \Delta \phi-nG'\phi \le 0.
\end{align*}
This forces that $\phi(p)=0$, and this leads to a contradiction. Therefore, $\phi(p)=0$, and thus $\phi \equiv 0$ by classical constant rank theorem.
\end{proof}


\begin{thebibliography}{10}
\bibitem{AC20} L. Ambrosio and X. Cabr\'e, Entire solutions of semilinear elliptic equations in $\mathbb{R}^3$ and a conjecture of De Giorgi, Journal Amer. Math. Soc. 13 (2000), 725-739.

\bibitem{BEL} W. Bergweiler, A. Eremenko and J. Langley. On conformal metrics of constant positive curvature in the plane. Journal of mathematical physics, analysis and geometry, 19, 1 (2023) 59-73.

\bibitem{BG} B. Bian and P. Guan, A microscopic convexity principle for nonlinear partial differential equations, Invent. Math., 177 (2009), 307-335.

\bibitem{BGMX} B. Bian, P. Guan, X. Ma and L. Xu, A constant rank theorem for quasiconcave solutions of fully nonlinear partial differential equations, Indiana Univ. Math. J. 60 No. 1 (2011), 101-120.

\bibitem{BIS23} P. Bryan, M. Ivaki, and J. Scheuer, Constant rank theorems for curvature problems via a viscosity approach,
Calc. Var. Partial Differential Equations 62 (2023), no. 3, Paper No. 98, 19 pp.

\bibitem{CGM} L. Caffarelli, P. Guan and X.-N. Ma, A constant rank theorem for solutions of fully nonlinear elliptic equations, Comm. Pure Appl. Math., 60 (2007), 1769-1791.

\bibitem{CF} L. Caffarelli and A. Friedman, Convexity of solutions of some semilinear elliptic equations,
Duke Math. J., 52 (1985), 431-455.

\bibitem{CL} W. Chen and C. Li, Classification of solutions of some nonlinear elliptic equations. Duke Math. J. 63(3) (1991), 615-622.

\bibitem{DKW} M. del Pino, M. Kowalczyk and J. Wei, On De Giorgi's conjecture in dimension $N \ge 9$. Ann. of Math. 174, 1485-1569 (2011).

\bibitem{EGLX} A. Eremenko, C. Gui, Q. Li and L. Xu. Rigidity results on Liouville equation. arXiv preprint (2022). arXiv:2207.05587.

\bibitem{Farina} A. Farina, Stable solutions of $-\Delta u = e^u$ on $\mathbb{R}^N$ , C. R. Math. Acad. Sci. Paris 345 (2007),
63-66. MR2343553 (2008e:35063).

\bibitem{DF} E.N. Dancer and A. Farina, On the classification of solutions of $-\Delta u=e^u$ on $\mathbb{R}^n$: stability outside a compact set and applications. Proc. Am. Math. Soc. 137(4), (2009), 1333-1338.

\bibitem{GG98} N. Ghoussoub and C. Gui, On a conjecture of de Giorgi and some related problems.
Math. Ann. 311 (1998) 481-491.

\bibitem{GLM} P. Guan, C. S. Lin and X.-N. Ma, The Christoffel-Minkowski problem II: Weingarten curvature
equations, Chin. Ann. Math., Ser. B, 27 (2006), 595-614.

\bibitem{GM} P. Guan and X.-N. Ma, The Christoffel-Minkowski problem I: Convexity of solutions of a
Hessian equations, Invent. Math., 151 (2003), 553-577.

\bibitem{GMZ} P. Guan, X.-N. Ma and F. Zhou, The Christoffel-Minkowski problem III: existence and convexity of admissible solutions, Comm. Pure Appl. Math., 59 (2006), 1352-1376.

\bibitem{GX} P. Guan and L. Xu, Convexity estimates for level sets of quasiconcave solutions to fully nonlinear elliptic equations, Journal fur die reine und angewandte Mathematik reine angew. 680 (2013), 41-67.

\bibitem{GL} C. Gui and Q. Li, Some geometric inequalities related to Liouville equations. Preprint.


\bibitem{KL87} N. J. Korevaar and J. L. Lewis, Convex solutions of certain elliptic equations have constant rank Hessians, Arch. Rational Mech. Anal., 97 (1987), 19-32.

\bibitem{MX08} X.-N. Ma and L. Xu, The convexity of solutions of a class of Hessian equation in bounded convex domain in $\mathbb{R}^3$, J. Funct. Anal., 255 (2008), 1713-1723.

\bibitem{Savin} O. Savin, Regularity of flat level sets in phase transitions, Ann. of Math. (2) 169 (2009), no. 1, 41-78. MR 2480601.

\bibitem{SX21} F. Sun and L. Xu, A quantitative constant rank theorem for quasiconcave solutions to fully nonlinear elliptic equations, Journal of Differential Equations 317 (2022), 685-705.

\bibitem{SW16} G. Sz\'ekelyhidi, B. Weinkove, On a constant rank theorem for nonlinear elliptic PDEs, Discrete Contin. Dyn. Syst. 36 (2016), 6523V6532.


\bibitem{SW21} G. Sz\'ekelyhidi, B. Weinkove, Weak Harnack inequalities for eigenvalues and constant rank theorems, Commun. Partial Differ. Equ. 46 (2021), 1585V1600.

\bibitem{Wang2017} K. Wang, A new proof of Savin's theorem on Allen-Cahn equations, Journal of the European Mathematical Society 19 (2017), 2997-3051.

\bibitem{wangwei} J. Wang and J. Wei. Finite Morse index implies finite ends. Communications on Pure and Applied Mathematics. 2019 May;72(5), 1044-119.



\end{thebibliography}
\end{document}